\documentclass[reqno]{amsart}
\usepackage{amsmath, amssymb, amsthm, epsfig}
\usepackage{hyperref, latexsym}
\usepackage{url}
\usepackage[mathscr]{euscript}

\usepackage{color}
\usepackage{fullpage} 
\usepackage{setspace}

\newcommand{\new}[1]{\textcolor{black}{#1}}

\def\today{\ifcase\month\or
  January\or February\or March\or April\or May\or June\or
  July\or August\or September\or October\or November\or December\fi
  \space\number\day, \number\year}
\DeclareMathOperator{\sgn}{\mathrm{sgn}}

\DeclareMathOperator{\supp}{\mathrm{supp}}

 \newtheorem{theorem}{Theorem}
  
 \newtheorem{lemma}[theorem]{Lemma}

 \theoremstyle{definition}
 
 \newtheorem{example}[theorem]{Example}
 \theoremstyle{remark}

 \newcommand{\R}{\mathbb{R}}

 \newcommand{\hh}{\tfrac12}

  \renewcommand{\d}{\text{\rm d}}
 \newcommand{\du}{\text{\rm d}u}

 \newcommand{\dx}{\text{\rm d}x}

\newcommand{\im}{{\rm Im}\,}
\newcommand{\re}{{\rm Re}\,}

\begin{document}
\title[A note on the zeros of zeta and $L$-functions]{A note on the zeros of zeta and $L$-functions}
\author[Carneiro, Chandee and Milinovich]{Emanuel Carneiro, Vorrapan Chandee and Micah B. Milinovich}
%\date{\today}
\subjclass[2000]{11M06, 11M26, 11M36, 11M41, 41A30}
\keywords{Riemann zeta-function, automorphic $L$-functions, Beurling-Selberg extremal problem, extremal functions, exponential type}
\address{IMPA - Instituto Nacional de Matem\'{a}tica Pura e Aplicada - Estrada Dona Castorina, 110, Rio de Janeiro, RJ, Brazil 22460-320}
\email{carneiro@impa.br}
\address{Department of Mathematics, Burapha University, 169 Long-Hard Bangsaen Road, Saen Sook Sub-district, Mueang District, Chonburi 20131, Thailand}
\email{vorrapan@buu.ac.th}
\address{Department of Mathematics, University of Mississippi, University, MS 38677 USA}
\email{mbmilino@olemiss.edu}

\allowdisplaybreaks
\numberwithin{equation}{section}

\maketitle

\begin{abstract}
Let $\pi S(t)$ denote the argument of the Riemann zeta-function at the point $s=\tfrac12+it$. Assuming the Riemann hypothesis, we give a new and simple proof of the sharpest known bound for $S(t)$. We discuss a generalization of this bound for a large class of $L$-functions including those which arise from cuspidal automorphic representations of GL($m$) over a number field.  We also prove a number of  related results including bounding the order of vanishing of an $L$-function at the central point and bounding the height of the lowest zero of an $L$-function.
\end{abstract}

\section{Introduction}

Let $\zeta(s)$ denote the Riemann zeta-function and, if $t$ does not correspond to an ordinate of a zero of $\zeta(s)$, let
\[
S(t) = \frac{1}{\pi} \arg \zeta(\tfrac{1}{2}\!+\!it),
\]
where the argument is obtained by continuous variation along straight line segments joining the points $2, 2+it$, and $\tfrac12+it$, starting with the value zero. If $t$ does correspond to an ordinate of a zero of $\zeta(s)$, set
\[
S(t) =\frac{1}{2}\,\lim_{\varepsilon\to 0}   \big\{ S(t\!+\!\varepsilon) + S(t\!-\!\varepsilon)  \big\}.
\]
The function $S(t)$ arises naturally when studying the distribution of the nontrivial zeros of the Riemann zeta-function. For $t>0$, let $N(t)$ denote the number of zeros $\rho=\beta+i\gamma$ of $\zeta(s)$ with ordinates satisfying $0<\gamma \le t$ where any zeros with $\gamma =  t$ are counted with weight $\tfrac12$. Then, for $t \ge 1$, it is known that
\[
N(t) = \frac{t}{2\pi} \log \frac{t}{2\pi} - \frac{t}{2\pi} +\frac{7}{8} + S(t) + O\Big(\frac{1}{t}\Big).
\]

\smallskip

In 1924, assuming the Riemann hypothesis, Littlewood \cite{L}  proved that
\[
S(t) = O\!\left( \frac{\log t}{\log\log t} \right)
\]
as $t\to \infty$. The order of magnitude of this estimate has never been improved.  Using the Guinand-Weil explicit formula for the zeros of $\zeta(s)$ and the classical Beurling-Selberg majorants and minorants for characteristic functions of intervals, Goldston and Gonek \cite{GG} established Littlewood's estimate in the form
\begin{equation}\label{GG}
\left| S(t) \right| \le \left(\frac{1}{2} \!+\!o(1)\right) \frac{\log t}{\log\log t}
\end{equation}
as $t\to \infty$. 
The authors \cite{CCM} recently sharpened this inequality and proved the following theorem.

\begin{theorem}\label{thm1}
Assume the Riemann hypothesis. Then
\[
\left| S(t) \right| \le  \left(\frac{1}{4} \!+\!o(1)\right) \frac{\log t}{\log\log t} 
\]
for sufficiently large $t$, where the term of $o(1)$ is $O(\log \log \log t / \log \log t)$.
\end{theorem}

In this note, we give a new and much simpler proof of Theorem \ref{thm1}. As in the work of Goldston and Gonek mentioned above, we use the explicit formula in conjunction with the classical majorants and minorants of exponential type for characteristic functions of intervals that were constructed by Beurling and Selberg. In \cite{CCM}, two proofs of Theorem \ref{thm1} are given. The more direct of these proofs proceeds as follows. Assuming the Riemann hypothesis, for $t$ not corresponding to an ordinate of a zero of $\zeta(s)$, it is shown in \cite{CCM} that
\[
S(t) = \frac{1}{\pi} \sum_{\gamma} F(t\!-\!\gamma) + O(1) ,
\]
where
\begin{equation}\label{Def_F}
F(x) = \arctan\left( \frac{1}{x} \right) \!-\! \frac{x}{1+x^2},
\end{equation}
and the sum runs over the nontrivial zeros $\rho=\tfrac12+i\gamma$ of $\zeta(s)$ counted with multiplicity.  Drawing upon the work of Carneiro and Littmann \cite{CL}, upper and lower bounds for $S(t)$ are established by replacing $F(x)$ in the sum over zeros by (optimally chosen) real entire majorants and minorants of $F(x)$ and then applying the explicit formula. Using similar ideas, Chandee and Soundararajan \cite{CS} proved that 
\[
\log|\zeta(\tfrac{1}{2}\!+\!it)| \le \left( \frac{\log 2}{2} \!+\! o(1) \right) \frac{\log t}{\log\log t}
\] 
as $t\to \infty$, also under the assumption of the Riemann hypothesis. The theory of extremal functions of exponential type has also been recently used to study zero repulsion in families of elliptic curve $L$-functions \cite{Marshall} and  to improve the upper and lower bounds for the pair correlation of zeros of $\zeta(s)$ under the Riemann hypothesis \cite{CCLM}.

\smallskip

In \text Section \ref{L-functions}, we discuss how to extend Theorem \ref{thm1} to a large class of $L$-functions  in both analytic conductor aspect and in $t$-aspect.  We also prove a number of  related results including bounding the order of vanishing of an $L$-function at the central point and bounding the height of the lowest zero of an $L$-function. If an $L$-function is self-dual, our new proof of Theorem \ref{thm1} generalizes in a natural  way.  However, if an $L$-function is not self-dual, then our new proof does not generalize but we show that an analogue of Theorem \ref{thm1} is still possible using a modification of our previous approach in \cite{CCM}.

\section{Beurling-Selberg majorants and minorants}

For functions $R \in L^1(\mathbb{R})$, let 
\[
\widehat{R}(\xi) = \int_{-\infty}^{\infty} R(x) \, e^{-2\pi i x \xi} \, \mathrm{d}x
\]
denote the Fourier transform of $R$. In the 1930s, Beurling observed that the real entire functions
\[
H^\pm(z) = \left(\frac{\sin \pi z}{\pi}\right)^2
 \left\{ \sum_{m=-\infty}^{\infty} \frac{\sgn(m)}{(z\!-\!m)^2} + \frac{2}{z} \right\} \pm \left(\frac{\sin \pi z}{\pi z}\right)^2
 \]
satisfy the inequalities $H^-(x) \le \sgn(x) \le H^+(x)$ for all real $x$, have Fourier transforms $\widehat{H}^\pm(\xi)$ supported in $[-1,1]$, and satisfy
\[
\int_{-\infty}^\infty \left| H^\pm(x) \!-\! \sgn(x) \right| \mathrm{d}x = 1.
\]
Moreover, Beurling showed that among all real entire majorants and minorants for $\sgn(x)$ with Fourier transforms supported in $[-1,1]$, the functions $H^\pm(x)$ minimize the $L^1(\mathbb{R})$-distance to $\sgn(x)$. The details of this construction can be found in \cite{V}.

\smallskip

For $t>0$, let $\chi_{[-t,t]}(x)$ denote the characteristic function of the interval $[-t,t]$ (normalized at the endpoints). From the above properties of Beurling's functions $H^\pm(z)$, Selberg observed that the real entire functions
\begin{equation}\label{Def_BS_majorants}
R^\pm(z) = \frac{1}{2} \, \Big\{ H^\pm\big(\Delta(t\!+\!z)\big) +H^\pm\big(\Delta(t\!-\!z)\big)  \Big\}
\end{equation}
satisfy the inequalities $R^-(x) \le \chi_{[-t,t]}(x) \le R^+(x)$ for all real $x$ and have Fourier transforms supported in $[-\Delta,\Delta]$. We summarize these (and some other relevant) properties of the functions $R^\pm(z)$ in the following lemma.

\begin{lemma} \label{Lem2}
Let $t$ and $\Delta$ be positive real numbers. The even real entire functions $R^\pm(z)$ defined in \eqref{Def_BS_majorants} satisfy the following properties:
\begin{enumerate}
\item[(i)] $\displaystyle{R^-(x) \le \chi_{[-t,t]}(x) \le R^+(x)}$ for all real $x$ \textup{;} 
\smallskip
\item[(ii)] $\displaystyle{\int_{-\infty}^{\infty} \big| R^\pm(x) -\chi_{[-t,t]}(x) \big| \, \mathrm{d}x = \frac{1}{\Delta} }$ \textup{;}
\smallskip
\item[(iii)] $\displaystyle{R^\pm(z) \ll e^{2\pi \Delta |\mathrm{Im } \, z|}}$ \textup{;}
\smallskip
\item[(iv)] $\displaystyle{R^\pm(x) \ll  \min(1, \Delta^{-2} (|x|-t)^{-2})}$ for $|x| > t$ \textup{;}
\smallskip
\item[(v)] $\displaystyle{\widehat{R}^\pm(\xi) = 0}$ for $|\xi| \ge \Delta$ \textup{;}
\smallskip
\item[(vi)] $\displaystyle{\widehat{R}^\pm(\xi) = \frac{\sin 2\pi t \xi}{\pi \xi} + O\Big(\frac{1}{\Delta}\Big)}$ for $|\xi| \le \Delta$.
\end{enumerate}
The implied constants in {\rm (iii), (iv)} and {\rm (vi)} are absolute.
\end{lemma}
\begin{proof}
This is a specialization of \cite[Lemma 2]{GG}. For a further discussion on these Beurling-Selberg extremal functions, as well as proofs of the above assertions, see the survey articles of Montgomery \cite{M}, Selberg \cite{S} and Vaaler \cite{V}.
\end{proof}

\section{Proof of Theorem \ref{thm1}}

The functional equation for the Riemann zeta-function, in symmetric form, can be written as
\[
\Gamma_{\mathbb{R}}(s)\zeta(s) =\Gamma_{\mathbb{R}}(1\!-\!s) \zeta(1\!-\!s),
\]
where $\Gamma_{\mathbb{R}}(s) = \pi^{-\frac{s}{2}} \Gamma(\tfrac{s}{2})$. 

\begin{lemma} \label{Lem3}
For all $t > 0$, we have
\[
N(t)  \, = \, \frac{1}{2\pi }\int_{-t}^t \frac{\Gamma_{\mathbb{R}}'}{\Gamma_{\mathbb{R}}}(\tfrac{1}{2}\!+\!iu) \, \mathrm{d}u +  \, S(t) +  1.
\]
\end{lemma}

\begin{proof}
Using the functional equation for $\zeta(s)$, the argument principle, and the Schwarz reflection principle, it can be shown that 
\[
N(t) = \frac{1}{\pi} \arg\Gamma_{\mathbb{R}}(\tfrac{1}{2}\!+\!it) +  S(t)  +  1
\]
for any $t > 0$.  See, for instance, \cite[Theorem 14.1]{MV}. Since
\[
\int_{-t}^t \frac{\Gamma_{\mathbb{R}}'}{\Gamma_{\mathbb{R}}}(\tfrac{1}{2}\!+\!iu) \,\mathrm{d}u = 2 \arg\Gamma_{\mathbb{R}}(\tfrac{1}{2}\!+\!it),
\]
the lemma follows.
\end{proof}

\begin{lemma} \label{Lem4}
Let $h(s)$ be analytic in the strip $|\mathrm{Im} \, s| \le \tfrac12+\varepsilon$ for some $\varepsilon>0$, and assume that $|h(s)| \ll (1+|s|)^{-(1+\delta)}$ for some $\delta>0$ when $|\mathrm{Re} \, s| \to \infty$.
Then
\begin{align}\label{Explicit_Formula}
\begin{split}
\sum_\rho h\!\left(\frac{\rho\!-\!\tfrac12}{i}\right) &= h\!\left(\frac{1}{2i}\right) + h\!\left(-\frac{1}{2i}\right) + \frac{1}{\pi}  \int_{-\infty}^{\infty} h(u) \, \mathrm{Re} \,  \frac{\Gamma_{\mathbb{R}}'}{\Gamma_{\mathbb{R}}}(\tfrac{1}{2}\!+\!iu) \, \mathrm{d}u 
\\
& \qquad \qquad - \frac{1}{2\pi}  \sum_{n=2}^\infty \frac{\Lambda(n)}{\sqrt{n}} \left\{   \widehat{h}\!\left( \frac{\log n}{2\pi} \right) + \widehat{h}\!\left( -\frac{\log n}{2\pi} \right)  \right\},
\end{split}
\end{align}
where the sum on the left-hand side runs over the nontrivial zeros $\rho$ of $\zeta(s)$ and $\Lambda(n)$ is the von Mangoldt function defined to be $\log p$ if $n=p^k$, $p$ a prime and $k\ge 1$, and zero otherwise. 
\end{lemma}
\begin{proof}
This is \cite[Lemma 1]{GG}. For a similar formula, see \cite[Equation (25.10)]{IK}.
\end{proof}

We now deduce Theorem \ref{thm1} from Lemmas \ref{Lem2}, \ref{Lem3} and \ref{Lem4}.

\begin{proof}[Proof of Theorem \ref{thm1}]
Assume the Riemann hypothesis and that $\Delta \ge 1$ and $t \ge 10$. By Lemma \ref{Lem2} (i), we have
\begin{equation} \label{proof1}
0 \le \sum_\gamma \Big\{ R^+(\gamma) - \chi_{[-t,t]}(\gamma) \Big\} = \sum_\gamma R^+(\gamma) - 2 \, N(t),
\end{equation}
where the sum runs over the ordinates of the nontrivial zeros $\rho=\tfrac12+i\gamma$ of $\zeta(s)$. Note that $R^+(u)$ is even and satisfies the conditions of Lemma \ref{Lem4}. In fact, the growth condition required in Lemma \ref{Lem4} follows by Lemma \ref{Lem2} (iii)-(iv) via the Phragm\'{e}n-Lindel\"{o}f principle. Using inequality \eqref{proof1} in conjunction with Lemma \ref{Lem3} and \eqref{Explicit_Formula}, we derive that
\begin{equation} \label{proof2}
S(t) \le  \frac{1}{2\pi}   \int_{-\infty}^{\infty} \Big\{ R^+(u) \!-\! \chi_{[-t,t]}(u) \Big\} \,  \frac{\Gamma_{\mathbb{R}}'}{\Gamma_{\mathbb{R}}}(\tfrac{1}{2}\!+\!iu) \, \mathrm{d}u - \frac{1}{2\pi}  \sum_{n=2}^\infty \frac{\Lambda(n)}{\sqrt{n}}  \, \widehat{R}^+\!\left( \frac{\log n}{2\pi} \right) +  R^+\!\left(\frac{1}{2i}\right) -1.
\end{equation}
Lemma \ref{Lem2} (iii) implies that $R^+\!\left(\frac{1}{2i}\right) \ll  e^{\pi \Delta}$, while Lemma \ref{Lem2} (v)-(vi) imply that the sum over prime powers in \eqref{proof2} is
\[
 \ll \sum_{n \le e^{2\pi \Delta}} \frac{\Lambda(n)}{\sqrt{n}} \left\{ \frac{1}{\log n} + \frac{1}{\Delta} \right\} \ll \frac{1}{\Delta}e^{\pi \Delta}\,
\]
by the Prime Number Theorem and partial summation. Therefore
\begin{equation} \label{proof3}
S(t) \le  \frac{1}{2\pi}   \int_{-\infty}^{\infty} \Big\{ R^+(u) \!-\! \chi_{[-t,t]}(u) \Big\} \,  \frac{\Gamma_{\mathbb{R}}'}{\Gamma_{\mathbb{R}}}(\tfrac{1}{2}\!+\!iu) \, \mathrm{d}u + O\!\left( e^{\pi \Delta} \right).
\end{equation}
It remains to estimate the integral in \eqref{proof3}. Note that, apart from the gamma factor, this is precisely the integral that Beurling and Selberg set out to minimize. Stirling's formula for the gamma function implies that
\[
\frac{\Gamma_\mathbb{R}'}{\Gamma_\mathbb{R}}\big(\tfrac{1}{2}\pm iu \big) = \frac{1}{2} \log \frac{u}{2\pi} + O\Big(\frac{1}{u}\Big)
\]
as $u \to \infty$. Hence, by Lemma \ref{Lem2} (iv), since $\Delta \ge 1$, we see that
\[
 \int_{2t}^\infty \Big\{ R^+(u) \!-\! \chi_{[-t,t]}(u) \Big\} \,  \frac{\Gamma_{\mathbb{R}}'}{\Gamma_{\mathbb{R}}}(\tfrac{1}{2}\!+\!iu) \, \mathrm{d}u \ll \int_{2t}^\infty \frac{\log u}{\Delta^2 (u\!-\!t)^2} \, \mathrm{d}u \ll \frac{\log t}{\Delta^2 \, t} \ll 1.
\]
Similarly, we have
\[
\int_{-\infty}^{-2t} \Big\{ R^+(u) \!-\! \chi_{[-t,t]}(u) \Big\} \,  \frac{\Gamma_{\mathbb{R}}'}{\Gamma_{\mathbb{R}}}(\tfrac{1}{2}\!+\!iu) \, \mathrm{d}u \ll \frac{\log t}{\Delta^2 \, t} \ll 1.
\]
Thus, by Lemma \ref{Lem2} (ii), it follows that
\begin{equation*}
\begin{split}
\int_{-\infty}^{\infty} \Big\{ R^+(u) \!-\! \chi_{[-t,t]}(u) \Big\} \,  \frac{\Gamma_{\mathbb{R}}'}{\Gamma_{\mathbb{R}}}(\tfrac{1}{2}\!+\!iu) \, \mathrm{d}u &= \int_{-2t}^{2t} \Big\{ R^+(u) \!-\! \chi_{[-t,t]}(u) \Big\} \,  \frac{\Gamma_{\mathbb{R}}'}{\Gamma_{\mathbb{R}}}(\tfrac{1}{2}\!+\!iu) \, \mathrm{d}u +O\!\left(1\right)
\\
& \le \left( \frac{\log t}{2} + O(1) \right) \int_{-2t}^{2t} \Big\{ R^+(u) \!-\! \chi_{[-t,t]}(u) \Big\} \, \mathrm{d}u +O\!\left(1\right)
\\
& \le \left( \frac{\log t}{2} + O(1) \right) \int_{-\infty}^{\infty} \Big\{ R^+(u) \!-\! \chi_{[-t,t]}(u) \Big\} \, \mathrm{d}u +O\!\left(1\right)
\\
& = \frac{\log t}{2\Delta} + O\!\left(\frac{1}{\Delta} + 1 \right).
\end{split}
\end{equation*}
Combining the above estimates, we find that
\[
S(t) \le \frac{\log t}{4\pi\Delta} + O\!\left( e^{\pi \Delta} \right).
\]
Choosing $\pi \Delta = \log\log t - 2\log\log \log t$, we deduce that
\[
S(t) \le \frac{1}{4} \frac{\log t}{\log \log t} + O\!\left( \frac{\log t \log\log\log t}{(\log\log t)^2} \right).
\]
This is the upper bound implicit in Theorem \ref{thm1}. 

\smallskip

Similarly, starting with the observation that
\[
 \sum_\gamma R^-(\gamma) - 2 \, N(t) \le 0,
\]
we can use Lemmas \ref{Lem2}, \ref{Lem3} and \ref{Lem4} to derive the inequality 
\[
S(t) \ge -\frac{\log t}{4\pi\Delta} + O\!\left( e^{\pi \Delta} \right).
\]
Again, choosing $\pi \Delta = \log\log t - 2\log\log \log t$, it follows that
\[
S(t) \ge -\frac{1}{4} \frac{\log t}{\log \log t} + O\!\left( \frac{\log t \log\log\log t}{(\log\log t)^2} \right).
\]
This completes the proof of Theorem \ref{thm1}.
\end{proof}

A remarkable property of the Beurling-Selberg majorants and minorants for the characteristic function of an interval is that the $L^1(\mathbb{R})$-distance of the approximations, as described in Lemma \ref{Lem2} (ii), is independent of the length of the interval. For this reason, the above proof essentially works equally well in bounding the number of zeros of $\zeta(s)$ with ordinates in the intervals $[-t,t]$ as it does for bounding the number of zeros with ordinates in $[0,t]$. Therefore, since $\zeta(s)$ has exactly the same number of zeros with imaginary parts in the intervals $[-t,0]$ and $[0,t]$, we are able to use the Beurling-Selberg majorants and minorants to bound $S(t)-S(-t)=2S(t)$ just as sharply as we are able to bound $S(t)-S(0) = S(t)$. This is, apart from a few technical issues, the key point which enables us to improve Goldston and Gonek's bound for $S(t)$ in \eqref{GG} by a factor of 2. These observations also allow us to generalize the proof of Theorem \ref{thm1} to self-dual $L$-functions. However, in order to generalize Theorem \ref{thm1} to $L$-functions which are not self-dual, it seems that the techniques in our previous paper \cite{CCM} are necessary. We have chosen to give the proof for the zeros of $\zeta(s)$ separately, and in full detail, in order to illustrate the simplicity of the method.

\section{Extension to general $L$-functions}\label{L-functions}

In this section, we discuss how to prove an analogue of Theorem \ref{thm1} for a large class of $L$-functions \new{(see for instance \cite[Chapter 5]{IK})}.  Although our results are stated more generally, the $L$-functions we have in mind arise from automorphic representations of GL($m$) over a number field.

\smallskip 

Let $m$ be a natural number. For $\mathrm{Re}(s)>1$, we suppose that an $L$-function, $L(s,\pi)$, of degree $m$ is given by the absolutely convergent Dirichlet series and Euler product of the form
\[
L(s,\pi) = \sum_{n=1}^\infty \frac{\lambda_\pi(n)}{n^s} = \prod_p \prod_{j=1}^m \left(1-\frac{\alpha_{j,\pi}(p)}{p^s} \right)^{\!-1},
\]
where $\lambda_\pi(1)=1$, $\lambda_\pi(n)\in \mathbb{C}$, and the local parameters $\alpha_{j,\pi}(p)$ are complex numbers which satisfy
\new{\[
|\alpha_{j,\pi}(p)| \le p^{\vartheta}
\]
for a constant $0\le \vartheta \le 1$}. %\footnote{\new{This constant may depend only on $m$ for some families, for instance when $L(s,\pi)$ arises from an irreducible cuspidal representation of GL($m$), as shown in \cite{LRS}.}}. 
We define the completed $L$-function, $\Lambda(s,\pi)$, by
\[
\Lambda(s,\pi) = L(s,\pi_\infty) \, L(s,\pi),
\]
where
\[
L(s,\pi_\infty) = N^{s/2} \prod_{j=1}^m \Gamma_\mathbb{R}(s \!+\! \mu_j).
\]
Here, as above, $\Gamma_{\mathbb{R}}(s) = \pi^{-\frac{s}{2}} \Gamma(\tfrac{s}{2})$, the natural number $N$ denotes the conductor of $L(s,\pi)$, and the spectral parameters $\mu_j$ are complex numbers which are either real or come in conjugate pairs and are \new{assumed to satisfy the inequality $\mathrm{Re}\,(\mu_j) > -1$.}\footnote{If $L(s,\pi)$ arises from an irreducible cuspidal representation of GL($m$),  then by \cite{LRS} we have $\mathrm{Re}\,(\mu_j) \ge -\frac12 + \frac{1}{m^2+1}$.} We assume that the completed $L$-function $\Lambda(s,\pi)$ extends to a meromorphic function of order 1 in $\mathbb{C}$ with at most poles at $s=0$ and $s=1$, each of order less than or equal to $m$. We further assume that $\Lambda(s,\pi)$ satisfies the functional equation 
\[
\Lambda(s,\pi) = \kappa\, \Lambda(1\!-\!s,\tilde{\pi}),
\]
where the root number $\kappa$ is a complex number of modulus 1 and $\Lambda(s,\tilde{\pi}) = \overline{\Lambda(\bar{s},\pi)}.$ Hence, by our assumptions on the complex numbers $\mu_j$, we have $L(s,\tilde{\pi}) = \overline{L(\bar{s},\pi)}$ and $L(s,\tilde{\pi}_\infty) =L(s,\pi_\infty)$. If $L(s,\tilde{\pi})=L(s,\pi)$, then we say that $L(s,\pi)$ is {\it self-dual}. Note that, by the functional equation, the (possible) poles of $\Lambda(s,\pi)$ at $s=0$ and $s=1$ have the same order which we denote by $r(\pi)$. As stated above, we have $0\le r(\pi) \le m$.

\smallskip

For $t>0$, let $N(t,\pi)$ denote the number of zeros $\rho_\pi=\beta_\pi+i\gamma_\pi$ of $\Lambda(s,\pi)$ which satisfy $0\le \beta_\pi \le 1$ and $-t \le \gamma_\pi \le t$, where any zeros with $\gamma_\pi = \pm t$ are counted with weight $\tfrac12$.\footnote{ Note the difference between the definition of $N(t,\pi)$ and the definition for $N(t)$ given in the introduction. If $L(s,\pi)=\zeta(s)$, then $N(t,\pi) = 2 N(t)$.} When $t$ is not an ordinate of a zero of $\Lambda(s,\pi)$, a standard application of the argument principle gives
\new{\begin{equation} \label{general_N}
N(t,\pi) = \frac{1}{\pi} \int_{-t}^t \re \frac{L'}{L}\big(\tfrac{1}{2} \!+\! iu,\pi_\infty \big) \, \mathrm{d}u + S(t,\pi) + S(t,\tilde{\pi}) + 2\, r(\pi) + O(m),
\end{equation}}
where 
\begin{equation}\label{S def1}
S(t,\pi) = \frac{1}{\pi} \arg L\big(\tfrac{1}{2} \!+\! it,\pi \big) = - \frac{1}{\pi} \int_{1/2}^\infty \mathrm{Im} \frac{L'}{L}\big(\sigma \!+\! it,\pi \big) \, \mathrm{d}\sigma
\end{equation}
and the term $O(m)$ corresponds to the contribution of the poles of $L(s, \pi_{\infty})$ when $-1 < \re(\mu_j) \leq -\frac12$. Generically this contribution is equal to $-2\,\#\big\{\mu_j:\, -1 < \re(\mu_j) < -\frac12\big\} - \#\big\{\mu_j: \, \re(\mu_j) = -\frac12\big\}$.
If $t$ does correspond to an ordinate of a zero of $\Lambda(s,\pi)$, we define
\begin{equation}\label{S def2}
S(t,\pi) = \frac{1}{2} \lim_{\varepsilon \to 0 }\big\{ S(t\!+\!\varepsilon,\pi)+S(t\!-\!\varepsilon,\pi) \big\}.
\end{equation}
Then the formula in \eqref{general_N} holds for all $t > 0$. We now prove an analogue of Theorem \ref{thm1} for the functions $S(t,\pi)$ in terms of the quantity
\[
C(t,\pi) = C(\pi) \, \big(|t|+1\big)^m ,
\]
where $C(\pi)=C(0,\pi)$ is often called the {\it analytic conductor} of $\Lambda(s,\pi)$ and is defined by
\[
C(\pi) = N \, \prod_{j=1}^m \big(|\mu_j| + 3 \big).
\]

\begin{theorem} \label{thm2}
Assume the generalized Riemann hypothesis for $\Lambda(s,\pi)$. Then, for all $t > 0$, we have
\begin{equation} \label{2.1}
|S(t,\pi)| \, \le \, \new{\left( \frac{1}{4} + \frac{\vartheta}{2} \right) }\frac{\log C(t,\pi)}{\log\log [C(t,\pi)^{3/m}]} + O\!\left( \frac{\log C(t,\pi) \log \log \log  [C(t,\pi)^{3/m}]}{\big(\log\log [C(t,\pi)^{3/m}]\big)^2}\right)
\end{equation}
and, consequently, we have
\begin{equation} \label{2.2}
\underset{s=\frac12}{\mathrm{ord}} \ \Lambda(s,\pi) \, \le \, \new{\left( \frac{1}{2} + \vartheta \right)} \frac{\log C(\pi)}{\log\log [C(\pi)^{3/m}]} + O\!\left( \frac{\log C(\pi) \log \log \log [C(\pi)^{3/m}] }{\big(\log\log [C(\pi)^{3/m}]\big)^2}\right).
\end{equation}
The implied constants above are universal.
\end{theorem}

\noindent{\sc Remarks.} 

(i) If $\pi$ is an irreducible cuspidal representation of GL($m$) over a number field, then Luo, Rudnick and Sarnak \cite{LRS} have shown that  $\vartheta \le \frac{1}{2} - \frac{1}{m^2+1}$. This bound can be improved when $m\le 4$ (see \cite{BB}). For instance, it is known that $\vartheta_2 =\frac{7}{64}$ from the work of Kim and Sarnak \cite{KS} and the corresponding generalization of their bound to number fields due to Nakasuji \cite{N} and to Blomer and Brumley \cite{BB}. The generalized Ramanujan-Petersson conjecture asserts that  $\vartheta=0$ for all $m\ge 1$. This conjecture is trivially true for the Riemann zeta-function and primitive Dirichlet $L$-functions, and is known for $L$-functions attached to classical holomorphic modular forms due to the work of Deligne \cite{D} and of Deligne and Serre \cite{DS}. 

\smallskip

(ii) If $L(s,\pi)$ is the Riemann zeta-function, then we recover Theorem \ref{thm1} from \eqref{2.1}. Further examples are discussed in Section \ref{examples}. 

\smallskip

(iii) If $L(s,\pi)$ is an $L$-function attached to an elliptic curve over $\mathbb{Q}$, then the estimate in \eqref{2.2} recovers the well-known result of Brumer \cite{B} on the analytic rank of an elliptic curve in terms of its conductor. In general, the estimate in \eqref{2.2} gives a sharpened form of \cite[Proposition 5.21]{IK}.

\smallskip

(iv)  By modifying the approach in \cite{CCM}, the analogue of the bound in \eqref{2.1} for the antiderivative 
\[
S_1(t,\pi) = \int_0^t S(u,\pi) \, \mathrm{d}u - \frac{1}{\pi} \int_{1/2}^\infty \log|L(\sigma,\pi)| \, \mathrm{d}\sigma
\]
is considered in \cite{Fin}.

\begin{proof}[Proof of Theorem \ref{thm2}]
By \eqref{S def1} and \eqref{S def2}, it suffices to prove the inequality \eqref{2.1} in the case where $t$ does not correspond to an ordinate of a zero of $\Lambda(s,\pi)$. We start by recording the explicit formula for $\Lambda(s,\pi)$, which can be established by modifying the proof of \cite[Theorem 5.12]{IK}. For functions $h$ satisfying the conditions in Lemma \ref{Lem4}, we have
 \begin{align}\label{Exp_Form_L}
\begin{split}
\sum_{\rho_\pi} h\!\left(\frac{\rho_\pi\!-\!\tfrac12}{i}\right) &= r(\pi) \left\{h\left(\frac{1}{2i} \right) +h\left(-\frac{1}{2i}\right) \right\} + \frac{1}{\pi}  \int_{-\infty}^{\infty} h(u) \, \mathrm{Re} \frac{L'}{L}\big(\tfrac{1}{2} \!+\! iu,\pi_\infty \big) \, \mathrm{d}u 
\\
& \qquad  \ \ \  - \frac{1}{2\pi}\sum_{n=2}^{\infty} \frac{1}{\sqrt{n}}\left\{ \Lambda_\pi(n)\, \widehat{h}\left(\frac{\log n}{2\pi }\right) + \Lambda_{\tilde{\pi}}(n)\, \widehat{h}\left( \frac{-\log n}{2\pi }\right)\right\}\\
& \qquad  \ \ \ - \sum_{-1<\re(\mu_j)< -\frac12} \left[ h\!\left(\frac{-\mu_j- \tfrac12}{i}\right)  + h\!\left(\frac{\tfrac12 + \mu_j}{i}\right)\right] \\
& \qquad \ \ \ - \frac{1}{2} \sum_{\re(\mu_j) =  -\frac12} \left[ h\!\left(\frac{-\mu_j- \tfrac12}{i}\right)  + h\!\left(\frac{\tfrac12 + \mu_j}{i}\right)\right],
\end{split}
\end{align}

where the sum runs over the zeros $\rho_\pi=\beta_\pi+i\gamma_\pi$ of $\Lambda(s,\pi)$ and the coefficients $\Lambda_\pi(n)$, which are supported on primes powers, are defined via the Dirichlet series 
\[
-\frac{\mathrm{d}}{\mathrm{d}s} \log L(s,\pi) = -\frac{L'}{L}(s,\pi) = \sum_{n=1}^\infty \frac{\Lambda_\pi(n)}{n^s}
\]
which converges absolutely for $\mathrm{Re}(s)>1$. Logarithmically differentiating the Euler product, it can be shown that 
\new{\begin{equation} \label{coeff}
|\Lambda_\pi(n)| \le m \, \Lambda(n)\, n^{\vartheta},
\end{equation}}
where $\Lambda(n)$ is the von Mangoldt function defined in Lemma \ref{Lem4}.

\subsubsection*{The self-dual case} We first handle the case where $L(s,\pi)$ is self-dual to show how our proof for the zeros of $\zeta(s)$ in the previous section can be extended. \new{Let $\Delta \geq 1$ and consider the majorant $R^+$ given by Lemma \ref{Lem2}}. By Lemma \ref{Lem2} (i), assuming the generalized Riemann hypothesis for $\Lambda(s,\pi)$, we have
\[ 
0 \le \sum_{\gamma_\pi} \Big\{ R^+(\gamma_\pi) - \chi_{[-t,t]}(\gamma_\pi) \Big\} = \sum_{\gamma_{\pi}} R^+(\gamma_\pi) -  N(t,\pi).
\]
Since $S(t,\pi)=S(t,\tilde{\pi})$ and $R^+(z)$ is an even function, the formula for $N(t,\pi)$ in \eqref{general_N} and the explicit formula \new{in \eqref{Exp_Form_L}} imply that
\begin{equation*}
\begin{split}
S(t,\pi) \, &\le \,   \frac{1}{2\pi} \int_{-\infty}^{\infty} \left\{ R^+(u) - \chi_{[-t,t]}(u) \right\} \mathrm{Re} \frac{L'}{L}\big(\tfrac{1}{2} \!+\! iu,\pi_\infty \big) \, \mathrm{d}u  
\\
&\qquad \qquad - \frac{1}{2\pi}  \! \sum_{n \le e^{2\pi\Delta}} \frac{ \Lambda_\pi(n)}{\sqrt{n}} \, \widehat{R^+}\!\left(\frac{\log n}{2\pi }\right) \, + \, r(\pi) \left\{ R^+\!\left( \frac{1}{2i}\right) -1\right\} + \new{O\left( m e^{\pi \Delta}\right)}.
\end{split}
\end{equation*}
\new{In the computation above, the last two sums of \eqref{Exp_Form_L} are accounted in the term $O\left( m e^{\pi \Delta}\right)$ due to Lemma \ref{Lem2} (iii).} Using parts (ii), (iii) and (vi) of Lemma \ref{Lem2}, along with the estimate in \eqref{coeff}, it follows that
\begin{align*}
S(t,\pi) &  \le \, \frac{1}{2\pi} \int_{-\infty}^{\infty} \Big\{ R^+(u) - \chi_{[-t,t]}(u) \Big\} \, \mathrm{Re} \frac{L'}{L}\big(\tfrac{1}{2} \!+\! iu,\pi_\infty \big) \, \mathrm{d}u  + \new{O\!\left(m \, e^{\pi (1+2 \, \vartheta) \Delta} \right)}\\
& = \frac{1}{4\pi \Delta}\log\left(\frac{N}{\pi^m}\right) + \frac{1}{4\pi} \sum_{j=1}^m \int_{-\infty}^{\infty} \! \Big\{ R^+(u) - \chi_{[-t,t]}(u) \Big\} \,  \re \frac{\Gamma'}{\Gamma} \! \left(\frac{\tfrac 12 \!+\! iu \!+\! \mu_j}{2}\right)\du + \new{O\!\left(m \, e^{\pi (1+2 \, \vartheta) \Delta} \right)}.
\end{align*}
To estimate each integral term in the sum above, it is convenient to use Stirling's formula in the form
\begin{equation}\label{Stirling1}
\frac{\Gamma'}{\Gamma}(s) = \log(s \!+\! 1) - \frac{1}{s} + O(1),
\end{equation}
\new{which is valid for $\mathrm{Re} \, (s) > -\frac12$, say. Since $\mathrm{Re} \, (\mu_j) > -1$ for each $j$,  we note that $\mathrm{Re} \, (\tfrac 12 + iu + \mu_j) > -\frac12$}. Also, for each $j$, we break the integral into three ranges by writing 
\[
\int_{-\infty}^{\infty} \!\!\Big\{ R^+(u) - \chi_{[-t,t]}(u) \Big\} \, \mathrm{Re} \, \frac{\Gamma'}{\Gamma} \left(\frac{\tfrac 12 \!+\! iu \!+\! \mu_j}{2}\right)\du = I_1 + I_2 + I_3,
\]
where the range of integration in $I_1$ is over the interval $(-\infty, -(|t|+1)(|\mu_j|+3))$, the range of integration in $I_2$ is over the interval $[-(|t|+1)(|\mu_j|+3), (|t|+1)(|\mu_j|+3)]$, and the range of integration in $I_3$ is over the interval $((|t|+1)(|\mu_j|+3),\infty)$.

To estimate $I_3$, we use \eqref{Stirling1} and Lemma \ref{Lem2} (iv) to derive that
\begin{align*}
I_3 \ll \int_{(|t|+1)(|\mu_j|+3)}^{\infty} \frac{1}{\Delta ^2 (u-t)^2} \big( \log u + O(1)\big)\,\du \ll \frac{1}{\Delta^2}\frac{\log [(|t|+1)(|\mu_j|+3)]}{(|t|+1)(|\mu_j|+3)}.
\end{align*}
The same bound holds for $I_1$. It remains to estimate $I_2$. If we write $\mu_j = a_j + ib_j$, then some care is needed to handle the range of integration over the interval $J=[-b_j-1, -b_j + 1]$ since it may be the case that $a_j$ \new{is close or equal to $-\tfrac12$}. \new{We can estimate the contribution from the interval $J$ to $I_2$ by using \eqref{Stirling1} and observing that the quantity
\begin{equation}\label{Deal_re_part}
\mathrm{Re} \left(-\frac{1}{\tfrac12 + iu + \mu_j}\right) \, = \, -\frac{(\tfrac12 + a_j)}{(\tfrac12 + a_j)^2 + (u + b_j)^2}
\end{equation}
is integrable in $u$, since it is a Poisson kernel (in case $a_j = -\frac12$ this quantity is actually zero), and using the fact that $R^+(u) - \chi_{[-t,t]}(u)$ is uniformly bounded. From this observation, we see that the contribution to $I_2$ from the interval $J$ is $O(1)$.} Therefore, letting $I =[-(|t|+1)(|\mu_j|+3), (|t|+1)(|\mu_j|+3)]$, we have
\begin{align*}
I_2 & \, \leq \, \int_{I-J}  \Big\{ R^+(u) - \chi_{[-t,t]}(u) \Big\} \, \log\big[(|t|+1)(|\mu_j|+3) + \tfrac12 + |\mu_j|\big] \,\du + O(1)\\
& \, \leq \, \int_{-\infty}^{\infty} \Big\{ R^+(u) - \chi_{[-t,t]}(u) \Big\} \, \log \big[(|t|+1)(|\mu_j|+3)\big] \,\du + O(1)\\
& \, = \, \frac{\log[(|t|+1)(|\mu_j|+3)]}{\Delta} + O(1).
\end{align*}
Combining these estimates, we arrive at the inequality
\[
S(t,\pi) \, \le \,  \frac{\log C(t,\pi)}{4\pi \Delta} + O\!\left(\frac{\log C(t,\pi)}{ \Delta^2 }\right)  + \new{O\!\left(m \, e^{\pi (1+2 \, \vartheta) \Delta} \right)}.
\]
Choosing \new{$\pi \, (1\!+\! 2 \,\vartheta) \, \Delta = \log \log C(t,\pi)^{3/m} - 2 \log \log \log C(t,\pi)^{3/m}$}, we derive the upper bound implicit in \eqref{2.1}.\footnote{The power of $3/m$ can be replaced by $a/m$ for any fixed constant $a>e$ (so that $\log \log a > 0$).} 

Similarly, starting with the observation that
\new{\[
 \sum_{\gamma_\pi} R^-(\gamma_\pi) -  \, N(t,\pi) \le 0,
\]}
we can derive the inequality 
\[
S(t,\pi) \, \ge \,  -\frac{\log C(t,\pi)}{4\pi \Delta} + O\!\left(\frac{\log C(t,\pi)}{ \Delta^2 }\right)  + \new{O\!\left(m \, e^{\pi (1+2 \, \vartheta) \Delta} \right)}.
\]
Again choosing \new{$\pi \, (1 + 2 \,\vartheta) \, \Delta = \log \log C(t,\pi)^{3/m} - 2 \log \log \log C(t,\pi)^{3/m}$},  we can establish the lower bound implicit in \eqref{2.1}. This completes the proof of the bound for $S(t,\pi)$ in Theorem \ref{thm2} in the case where $L(s,\pi)$ is self-dual. We note that if $L(s,\pi)$ is not self-dual, then the above proof leads to the inequality
\[
\big|S(t,\pi) + S(t,\tilde\pi)  \big|  \, \le \, \new{\left( \frac{1}{2} + \vartheta \right)} \frac{\log C(t,\pi)}{\log\log [C(t,\pi)^{3/m}]} + O\!\left( \frac{\log C(t,\pi) \log \log \log  [C(t,\pi)^{3/m}]}{\big(\log\log [C(t,\pi)^{3/m}]\big)^2}\right),
\]
and it does not seem possible to derive the inequality in \eqref{2.1} from this estimate.

\subsubsection*{The general case} In order to prove \eqref{2.1} for general $L$-functions (not necessarily self-dual), we follow our alternative approach described in \cite{CCM}.  Note that 
\begin{equation}\label{Exp_S_L}
\begin{split}
S(t,\pi) %= - \frac{1}{\pi} \int_{1/2}^\infty \mathrm{Im} \frac{L'}{L}\big(\sigma+it,\pi \big) \, \mathrm{d}\sigma 
&= - \frac{1}{\pi} \int_{1/2}^{5/2} \mathrm{Im} \frac{L'}{L}\big(\sigma \!+\! it,\pi \big) \, \mathrm{d}\sigma + O(m) 
\\
&=  - \frac{1}{\pi} \int_{1/2}^{5/2} \im \left\{ \frac{L'}{L} (\sigma \!+\! i t, \pi) - \frac{L'}{L} (\tfrac52 \!+\! i t, \pi) \right\} \d\sigma + O(m).
\end{split}
\end{equation}
This estimate follows from \eqref{S def1}  and the coefficient bound in \eqref{coeff} since
\begin{equation*}
\left|\int_{5/2}^\infty \mathrm{Im} \frac{L'}{L}\big(\sigma \!+\! it,\pi \big) \, \mathrm{d}\sigma\right| \leq \int_{5/2}^{\infty}\sum_{n \ge 1} \frac{|\Lambda_{\pi}(n)|}{n^{\sigma}} \, \mathrm{d}\sigma = \sum_{n \ge 1} \int_{5/2}^{\infty} \frac{|\Lambda_{\pi}(n)|}{n^{\sigma}} \, \mathrm{d}\sigma = \sum_{n \ge 1} \frac{|\Lambda_{\pi}(n)|}{n^{5/2}\log n} = O(m).
\end{equation*}
To estimate the second integral on the right-hand side of \eqref{Exp_S_L}, we use the partial fraction decomposition \cite[Equation (5.24)]{IK}
\begin{align}\label{Eq23}
- \frac{L'}{L}(s,\pi) = \frac12\log\left(\frac{N}{\pi^m}\right)+ \frac12 \, \sum_{j=1}^m \, \frac{\Gamma'}{\Gamma}\!\left( \frac{s + \mu_j}{2}\right) - B_\pi + \frac{r(\pi)}{s} + \frac{r(\pi)}{s-1} - \sum_{\rho_{\pi}} \left( \frac{1}{s-\rho_{\pi}} + \frac{1}{\rho_{\pi}}\right)
\end{align}
where $B_\pi$ is a complex number. Now insert \eqref{Eq23} into \eqref{Exp_S_L}. Evidently, the constant terms cancel. Moreover, it is not difficult to estimate the other terms which do not involve the zeros of $\Lambda(s,\pi)$.  Writing $s = \sigma + it$ with $\tfrac12 \leq \sigma \leq \tfrac52$, we have 
\[
\mathrm{Im} \left(\frac{r(\pi)}{s}\right) \, = \, O(m)
\]
and
\[
\mathrm{Im} \left(\frac{r(\pi)}{s-1}\right) \, = \, -r(\pi)\left[\frac{t}{(\sigma \!-\! 1)^2 + t^2}\right].
\]
The function on the right-hand side of the second expression is integrable in $\sigma$ as long as $t \neq 0$ (in fact, it is integrable from $\sigma=-\infty$ to $\sigma=\infty$ since it is a Poisson kernel). Thus, the contributions from these terms to \eqref{Exp_S_L} is $O(m)$. We use \eqref{Stirling1} to handle the terms involving the gamma factors. For $\tfrac12 \leq \sigma \leq \tfrac52$, we have \new{$\mathrm{Re} \, (s + \mu_j) > -\frac12$}. Thus, for $|s + \mu_j| >2$, it follows that
\[
\mathrm{Im} \, \frac{\Gamma'}{\Gamma}\left( \frac{s + \mu_j}{2}\right) \, = \, \new{\mathrm{Im} \,  \log \left( \frac{s + \mu_j}{2} + 1\right)} + O(1) \, = \, O(1).
\]
If $|s + \mu_j| \leq 2$, we write $\mu_j = a_j + ib_j$ and observe that
\[
\mathrm{Im} \, \frac{\Gamma'}{\Gamma}\left( \frac{s + \mu_j}{2}\right) \, = \, -\mathrm{Im}  \left( \frac{2}{s + \mu_j}\right) + O(1) \, = \, \frac{2 \, (t+b_j)}{(\sigma + a_j)^2 + (t + b_j)^2}+  \new{O(1).}
\]
Again, this is integrable in $\sigma$ since it is a Poisson kernel (and if $t = -b_j$ it is simply equal to $0$). Summing each of these factors from $j =1$ to $j=m$ we get another term of size $O(m)$. It therefore follows that
\begin{equation} \label{S_L_formula}
\begin{split}
S(t,\pi) &=  - \frac{1}{\pi} \int_{1/2}^{5/2} \im\left( \frac{L'}{L} (\sigma + it, \pi) - \frac{L'}{L} \big(\tfrac{5}{2} + it, \pi \big)\right)  \d\sigma + O(m)\\
&  = \frac{1}{\pi} \int_{1/2}^{5/2} \sum_{\gamma_{\pi}}\left\{ \frac{(t - \gamma_{\pi})}{(\sigma - \hh)^2 + (t - \gamma_{\pi})^2} - \frac{(t - \gamma_{\pi})}{4 + (t - \gamma_{\pi})^2}\right\} \d\sigma+ O(m)\\
&  = \frac{1}{\pi}\sum_{\gamma_{\pi}} \int_{1/2}^{5/2} \left\{ \frac{(t - \gamma_{\pi})}{(\sigma - \hh)^2 + (t - \gamma_{\pi})^2} - \frac{(t - \gamma_{\pi})}{4 + (t - \gamma_{\pi})^2}\right\} \d\sigma+ O(m)\\
& = \frac{1}{\pi} \sum_{\gamma_{\pi}} \left\{ \arctan \left( \frac{2}{(t-\gamma_{\pi})}\right) - \frac{2 \, (t-\gamma_{\pi})}{4 + (t - \gamma_{\pi})^2}\right\} + O(m)
\\
&  = \frac{1}{\pi}  \sum_{\gamma_{\pi}} G(t-\gamma_{\pi}) + O(m),
\end{split}
\end{equation}
where $G(x) = \arctan(2/x) - 2x/(4+x^2)$. We now use the extremal functions $m_{\Delta}^{\pm}(x)$ of exponential type $2\pi \Delta$ constructed in \cite[Lemma 6]{CCM}.\footnote{Note that $G(x) = F(x/2)$ for $F$ defined in \eqref{Def_F}. The extremal functions in \cite[Lemma 6]{CCM} must be dilated accordingly.} For $\Delta \geq 1$, these are real entire functions satisfying the following properties:
\begin{enumerate}
\item[(P1)] For all real $x$ we have
\[
-\frac{C}{1+x^2}\leq m_{\Delta}^-(x) \leq G(x) \leq m_{\Delta}^+(x) \leq \frac{C}{1+x^2}\,,
\]
where $C$ is a universal constant and, for all complex $z$, we have
\[
\big|m_{\Delta}^{\pm}(z)\big| \ll \frac{\Delta^2}{1 + \Delta |z|} e^{2 \pi \Delta |\im(z)|}.
\]

\item[(P2)] We have $\supp\big(\widehat{m}_{\Delta}^{\pm}\big) \subset[-\Delta, \Delta]$ and $\widehat{m}_{\Delta}^{\pm}(\xi) \ll 1$ for all $\xi \in [-\Delta, \Delta]$.
\smallskip
\item[(P3)] The $L^1(\R)$-distances to $G(x)$ are minimized, with values given by
\[
\int_{-\infty}^{\infty} \Big\{ m_{\Delta}^+(x) - G(x) \Big\} \,\dx \, = \, \int_{-\infty}^{\infty} \Big\{G(x) - m_{\Delta}^-(x) \Big\} \,\dx \, = \, \frac{\pi}{2\Delta}.
\]
\end{enumerate}
By \eqref{S_L_formula} and (P1), we have
\begin{equation*}
S(t,\pi) \, \leq \, \frac{1}{\pi}  \sum_{\gamma_{\pi}} m_{\Delta}^+(t-\gamma_{\pi}) + O(m).
\end{equation*}
Using the explicit formula \eqref{Exp_Form_L} with $h(z) = m_{\Delta}^+(t-z)$, the properties (P1) - (P3) above imply that
\begin{equation*}
S(t,\pi) \, \leq \, \frac{1}{4\pi \Delta}\log\left(\frac{N}{\pi^m}\right)  + \frac{1}{2\pi^2} \sum_{j=1}^m \int_{-\infty}^{\infty} m_{\Delta}^+(u)\, \re \frac{\Gamma'}{\Gamma} \left(\frac{\tfrac 12 + it - iu + \mu_j}{2}\right)\du + \new{O\!\left(m \,\Delta\, e^{\pi (1+2 \, \vartheta) \Delta} \right).}
\end{equation*}
For each $j$, we estimate the corresponding integral term using \eqref{Stirling1}. If $|\tfrac 12 + it - iu + \mu_j| \le 2$, we deal with this real part as in \eqref{Deal_re_part} and the contribution is $O(1)$. If $|\tfrac 12 + it - iu + \mu_j| > 2$, then the main term from \eqref{Stirling1} is equal to
\begin{equation*}
\begin{split}
\log \left(\tfrac32 + it - iu +\mu_j \right) \, &= \, \log \left(|it +\mu_j|+3\right) + \log \left(\frac{ \tfrac32 + it - iu +\mu_j}{(|it + \mu_j|+3)} \right)
\\
 \, &= \, \log \left(|it +\mu_j|+3\right) + O\big(\log(|u|+2)\big).
\end{split}
\end{equation*}
Integrating this against the majorant $m_\Delta^+(u)$ gives a term of size\footnote{\new{Note that this allows a slightly sharper version of Theorem \ref{thm2}, with conductor defined by
$$\widetilde{C}(t,\pi) =  N \, \prod_{j=1}^m \big(|it + \mu_j| + 3 \big).$$}}
\[
\frac{\pi}{2\Delta}  \log \left(|it +\mu_j|+3\right)+ O(1) \, \leq \, \frac{\pi}{2\Delta} \log \left[(|t|+1)(|\mu_j|+3)\right] + O(1).
\]
Summing over $j$ and combining estimates, we arrive at the inequality
\new{\[
S(t,\pi) \, \le \,  \frac{\log C(t,\pi)}{4\pi \Delta}  + O\!\left(m \, \Delta\, e^{\pi (1+2 \, \vartheta) \Delta} \right).
\]}
Choosing $\new{\pi \, (1\!+\! 2 \,\vartheta) \, \Delta} = \log \log C(t,\pi)^{3/m} - 3 \log \log \log C(t,\pi)^{3/m}$, we derive the upper bound implicit in \eqref{2.1}. The proof of the lower bound for $S(t,\pi)$  implicit in \eqref{2.1} is analogous.

\subsubsection*{Proof of \eqref{2.2}} Finally, the estimate in \eqref{2.2} follows from \eqref{general_N} and  \eqref{2.1} by noting that 
\[
\lim_{t\to 0} \, N(t,\pi) \, = \, \underset{s=\frac12}{\mathrm{ord}} \ \Lambda(s,\pi).
\]
This completes the proof of Theorem \ref{thm2}.
\end{proof}

The bound for $S(t,\pi)$ in Theorem \ref{thm2} can \new{be} improved in $t$-aspect for automorphic $L$-functions. In particular, let $\pi$ be an irreducible unitary cuspidal representation of $\mathrm{GL}(m)$ over $\mathbb{Q}$ and let $L(s,\pi)$ be the corresponding $L$-function. Setting $\Lambda_\pi(n) = \Lambda(n) a_\pi(n)$ and using the fact that the Rankin-Selberg $L$-function $L(s,\pi \times \tilde{\pi})$ has a simple pole at $s=1$, Rudnick and Sarnak \cite[Section 2.4]{RS} observed that
\[
\sum_{n\le x} \frac{\Lambda(n) |a_\pi(n)|^2 \log n}{n} \, \sim \, \frac{\log^2x}{2} 
\]
as $x \to \infty$.  Since 
\[
\sum_{n\le x} \frac{|\Lambda_\pi(n)|^2}{n} \le \sum_{n\le x} \frac{\Lambda(n) |a_\pi(n)|^2 \log n}{n}
\]
%\new{$|\Lambda_\pi(n)| = \Lambda(n) |a_\pi(n)|$ (is this trivial?) -- This is wrong! Correct this tomorrow.}, 
and the coefficients $\Lambda_\pi(n)$ are supported on prime powers, an application of Cauchy's inequality shows that
\[
 \sum_{n \le x} \frac{|\Lambda_\pi(n)|}{\sqrt{n}} = O\!\left( \sqrt{ x \log x} \, \right),
\]
where the implied constant depends on $\pi$ (and thus on $C(\pi)$ and $m$). Using this estimate in the above proof, we can derive the following result.

\begin{theorem}
Let $\pi$ be an irreducible unitary cuspidal representation of $\mathrm{GL}(m)$ over $\mathbb{Q}$ and let $L(s,\pi)$ be the corresponding $L$-function. Then, assuming the generalized Riemann hypothesis for $L(s,\pi)$, we have
\[
|S(t,\pi)| \, \le \, \frac{m}{4}\frac{\log t}{ \log \log t} + O\!\left( \frac{\log t \log\log \log t}{(\log\log t)^2} \right)
\]
when $t$ is sufficiently large, where the implied constant depends on $m$ and on the representation $\pi$.
\end{theorem}

\section{The lowest zero of an $L$-function}

In addition to bounding the order of vanishing at the central point, Theorem \ref{thm2} also allows us to bound the height of the lowest zero of a completed $L$-function satisfying the hypotheses of the previous section in terms of its analytic conductor. In particular, we prove the following theorem (which sharpens a recent result of Omar \cite{O} for automorphic $L$-functions).

\begin{theorem}
Assume the generalized Riemann hypothesis for $\Lambda(s,\pi)$. Then there is at least one zero $\rho_\pi = \frac{1}{2}+i\gamma_\pi$ of $\Lambda(s,\pi)$ satisfying
\[
|\gamma_\pi| \, \le \,  \new{\left( \frac{1}{2} + \vartheta \right)} \frac{ \pi}{\log\log [C(\pi)^{3/m}]} + O\!\left( \frac{ \log \log  \log C(\pi)^{3/m}}{\big(\log\log [C(\pi)^{3/m}]\big)^2}\right)
\]
when  $C(\pi)^{3/m}$ is sufficiently large, where the implied constant is universal.
\end{theorem}
\begin{proof} Observe that if $\Lambda(\frac{1}{2},\pi) = 0$, then the theorem holds. Thus, we may assume that $\Lambda(\frac{1}{2},\pi) \ne 0$. If $t$ is small, say $0 < t \le 1$, then by applying  \eqref{Stirling1} and \eqref{Deal_re_part} we can show that
\begin{equation*}
\begin{split}
\frac{1}{\pi} \int_{-t}^t \frac{L'}{L}\big(\tfrac{1}{2}+iu,\pi_\infty) \, \mathrm{d}u &= \frac{t}{\pi} \log\big(N/\pi^m\big) + \frac{1}{2\pi} \sum_{j=1}^m \int_{-t}^t \mathrm{Re} \frac{\Gamma'}{\Gamma}\left(\frac{\frac{1}{2}+iu+\mu_j}{2}\right)\, \mathrm{d}u
\\
&= \frac{t}{\pi} \log C(\pi) + O(m).
\end{split}
\end{equation*}
Using this estimate in \eqref{general_N} and observing that $r(\pi) \ge0$, it follows that
\begin{equation}\label{N_ineq}
N(t,\pi) \, \ge \, \frac{t}{\pi} \log C(\pi) + S(t,\pi) + S(t,\tilde{\pi}) + O(m)
\end{equation}
for $0< t \le 1$. When $0 \leq t \leq 1$ we have 
\begin{equation}\label{Sec5_upp_bound1}
\log C(\pi) \leq \log C(t, \pi) \leq m \log 2 + \log C(\pi)
\end{equation}
which implies that 
\begin{equation}\label{Sec5_upp_bound2}
\log \tfrac{3}{m} \log C(\pi)\leq \log \tfrac{3}{m}\log C(t, \pi)\leq \log \left[ 3 \log 2+ \tfrac{3}{m}  \log C(\pi)\right].
\end{equation}
Using \eqref{Sec5_upp_bound1} and \eqref{Sec5_upp_bound2} in Theorem \ref{thm2} we get
\begin{align*}
|S(t,\pi)| \,&  \le \new{ \left( \frac{1}{4} + \frac{\vartheta}{2} \right)} \frac{(m \log 2 + \log C(\pi))}{\log\log [C(\pi)^{3/m}]} +  O\!\left( \frac{(m \log 2 + \log C(\pi))\log \log \left[ 3 \log 2+ \tfrac{3}{m}  \log C(\pi)\right]}{\big(\log\log [C(\pi)^{3/m}]\big)^2}\right)\\
& = \new{\left( \frac{1}{4} + \frac{\vartheta}{2} \right)} \frac{ \log C(\pi)}{\log\log [C(\pi)^{3/m}]} +  O\!\left( \frac{ \log C(\pi)\log \log  \log C(\pi)^{3/m}}{\big(\log\log [C(\pi)^{3/m}]\big)^2}\right) + O(m).
\end{align*}
From \eqref{N_ineq} we arrive at
\begin{align*}
N(t,\pi) \geq \frac{t}{\pi} \frac{m}{3}\log C(\pi)^{3/m} - \new{\left( \frac{1}{2} + \vartheta \right)} \frac{m}{3} \frac{ \log C(\pi)^{3/m}}{\log\log [C(\pi)^{3/m}]} +  O\!\left( \frac{ m\log C(\pi)^{3/m}\log \log  \log C(\pi)^{3/m}}{\big(\log\log [C(\pi)^{3/m}]\big)^2}\right).
\end{align*}
The right-hand side of this inequality is positive when 
\begin{equation*}
t \geq \new{\left( \frac{1}{2} + \vartheta \right) }\frac{ \pi}{\log\log [C(\pi)^{3/m}]} + O\!\left( \frac{ \log \log  \log C(\pi)^{3/m}}{\big(\log\log [C(\pi)^{3/m}]\big)^2}\right),
\end{equation*}
and this concludes the proof.
\end{proof}

\section{Examples}\label{examples}

In this section, we state a few consequences of Theorem \ref{thm2}. For the definitions and relevant properties of the zeta and $L$-functions described in the following examples, see \cite[Chapter 5]{IK}.

\begin{example}
Let $q >2$ be an integer, let $\chi$ be a primitive Dirichlet character modulo $q$, and let $L(s,\chi)$ be the corresponding Dirichlet $L$-function. Then Theorem \ref{thm2} implies that
\[
|S(t,\chi)| \ \le \ \left(\frac{1}{4}+o(1) \right) \frac{\log q(|t|\!+\!1)}{\log \log q(|t|\!+\!1)} \quad \text{and} \quad  \underset{s=\frac12}{\mathrm{ord}} \ L(s,\chi) \ \le \ \left(\frac{1}{2}+o(1) \right) \frac{\log q}{\log \log q},
\]
assuming the generalized Riemann hypothesis for $L(s,\chi)$ where $S(t,\chi) = \frac{1}{\pi}\arg L(\frac{1}{2}+it,\chi)$.
\end{example}

\smallskip

\begin{example}
Let $f$ be a holomorphic newform of weight $k\ge1$ and level $q$, let
\[
f(z) = \sum_{n\ge 1} \lambda_f(n) \, n^{(k-1)/2} \, e^{2\pi i n z}
\]
be its normalized Fourier expansion at the cusp $\infty$, and let $L(s,f) =\sum_{n\ge 1} \lambda_f(n)  n^{-s}$ for $\mathrm{Re}(s)>1$ be the corresponding $L$-function. Then
\[
|S(t,f)| \ \le \ \left(\frac{1}{2}+o(1) \right) \frac{\log \sqrt{q}(k\!+\!3)(|t|\!+\!1)}{\log \log \sqrt{q}(k\!+\!3)(|t|\!+\!1)} \quad \text{and} \quad  \underset{s=\frac12}{\mathrm{ord}} \ L(s,f) \ \le \ \big( 1+o(1) \big) \frac{\log \sqrt{q}(k\!+\!3)}{\log \log \sqrt{q}(k\!+\!3)},
\]
assuming the generalized Riemann hypothesis for $L(s,f)$ where $S(t,f) = \frac{1}{\pi}\arg L(\frac{1}{2}+it,f)$.

\end{example}

\smallskip

\begin{example}
Let $K$ be a finite extension of $\mathbb{Q}$ with %absolute 
discriminant $d_K$ and degree $[K:\mathbb{Q}]$, and let $\zeta_K(s)$ denote the associated Dedekind zeta-function. Further let $\mathrm{rd}_K = |d_K|^{1/[K:\mathbb{Q}]}$ denote the root discriminant of $K$. Then
\[
|S_K(t)| \ \le \ \left(\frac{[K:\mathbb{Q}]}{4}+o(1) \right) \frac{\log \mathrm{rd}_K(|t|\!+\!3)}{\log \log \mathrm{rd}_K(|t|\!+\!3)} \quad \text{and} \quad  \underset{s=\frac12}{\mathrm{ord}} \ \zeta_K(s)\ \le \ \left(\frac{1}{2}+o(1) \right) \frac{\log(|d_K|\!+\!3)}{\log \log (|d_K|\!+\!3)},
\]
assuming the generalized Riemann hypothesis for $\zeta_K(s)$ where $S_K(t) = \frac{1}{\pi}\arg \zeta_K(\frac{1}{2}+it)$.
\end{example}

\section*{Acknowledgements.} 
\noindent EC acknowledges support from CNPq-Brazil grants $302809/2011-2$ and $477218/2013-0$, and FAPERJ grant $E-26/103.010/2012$. MBM is supported in part by an AMS-Simons Travel Grant and the NSA Young Investigator Grant H98230-13-1-0217. We would like to thank Renan Finder for valuable suggestions.

\end{document}